\documentclass[11pt]{article}

\usepackage{amsmath,amssymb,amsthm}
\newtheorem{theorem}{Theorem}[section]
\newtheorem{lemma}[theorem]{Lemma}
\newtheorem{proposition}[theorem]{Proposition}
\newtheorem{definition}[theorem]{Definition}
\newtheorem{corollary}[theorem]{Corollary}

\newtheorem{example}[theorem]{Example}
\numberwithin{equation}{section}

\def\({\bigl(}                \def\){\bigr)}

                        \def\^{\tilde}

\def\1{1\!\!1}

\title{Dynamical properties of $S$-gap shifts and other shift spaces}

\author{
Simon Baker\\
School of Mathematics\\
University of Manchester, Manchester, M13 9PL, UK.\\
email: simonbaker412@gmail.com
\\ \\and\\ \\
Andrei E. Ghenciu\\
Department of Mathematics, Statistics and Computer Science\\
University of Wisconsin Stout, Menomonie, WI, 54751, USA\\
email: ghenciue@uwstout.edu
\thanks{
The second author would like to thank 
the University of Alaska --- Fairbanks for its hospitality during
his stays while part of this paper was written.}
}
\date{\today}

\begin{document}

\maketitle

\begin{abstract}
We study the dynamical properties of certain shift spaces. To help study these properties we introduce two new classes of shifts, namely boundedly supermultiplicative (BSM) shifts and balanced shifts. It turns out that any almost specified shift is both BSM and balanced, and any balanced shift is BSM. However, as we will demonstrate, there are examples of shifts which are BSM but not balanced. We also study the measure theoretic properties of balanced shifts. We show that a shift space admits a Gibbs state if and only if it is balanced.

Restricting ourselves to $S$-gap shifts, we relate certain dynamical properties of an $S$-gap shift to combinatorial properties from expansions in non-integer bases. This identification allows us to use the machinery from expansions in non-integer bases to give straightforward constructions of $S$-gap shifts with certain desirable properties. We show that for any $q\in(0,1)$ there is an $S$-gap shift which has the specification property and entropy $q$. We also use this identification to address the question, for a given $q\in(0,1),$ how many $S$-gap shifts exist with entropy $q?$ For certain exceptional values of $q$ there is a unique $S$-gap shift with this entropy.
\end{abstract}

\section{Introduction}
In this paper we study the dynamical properties of certain shift spaces. We begin by giving a basic overview. For a more thorough account we refer the reader to \cite{LinMar}. A finite or countably infinite set, usually denoted $\mathcal{A}$, is called the \textbf{alphabet}. The elements of $\mathcal{A}$ will be called \textbf{letters}.
\begin{definition}
The \textbf{full $\mathcal{A}$-shift} is
$$\mathcal{A^{\mathbb{Z}}}:= \{ x = (...,x_{-1},x_0, x_1, ..., x_i,...): x_i \in \mathcal{A}\}$$
\end{definition}
The full $\mathcal{A}$-shift is just the collection of all bi-infinite sequences with letters from $\mathcal{A}$. A shift space and the shift map are defined as follows.
\begin{definition}
A \textbf{shift space} is a subset $X$ of $\mathcal{A^{\mathbb{Z}}}$ such that $$X:= X_{\mathcal{F}}=\Big\{x\in A^{\mathbb{Z}}: \textrm{ for all }n\in\mathbb{Z}, k\in\mathbb{N} \textrm{ we have }(x_{n},x_{n+1},\ldots,x_{n+k})\notin \mathcal{F}\Big\}.$$ For some collection $\mathcal{F}$ of forbidden blocks over $\mathcal{A}$. The \textbf{shift map} $\sigma:X\to X$ is defined on any shift space according to the rule $\sigma(x)=y$ where $y_i=x_{i+1}$ for each $i\in\mathbb{Z}.$
\end{definition}A \textbf{shift of finite type} is a shift space where the set of forbidden blocks $\mathcal{F}$ is finite. We can equip shift spaces with a metric. The metric captures the idea that points are close when large central blocks of their coordinates agree. If $x,y \in \mathcal{A^{\mathbb{Z}}}$ we put:
$$ d(x,y)=2^{-k}$$
where $k$ is maximal so that $(x_{-k},\ldots,x_{k})=(y_{-k},\ldots,y_{k}).$ A subset $X$ of $\mathcal{A^{\mathbb{Z}}}$ is a shift space if and only if it is shift-invariant and compact.
A key concept in the general theory of shift spaces is the concept of entropy, which we define next.
\begin{definition}
Let $X$ be a shift space over a finite alphabet. The \textbf{entropy} of $X$ is defined by:
$$h(X):= \lim_{n \to \infty} \frac{1}{n} \log |{\cal B}_n(X)|,$$
where ${\cal B}_n(X)$ is the number of $n$-blocks (blocks of length $n$) appearing in elements of $X$.
\end{definition}If $X$ has alphabet $\mathcal{A}$, then $h(X) \leq \log |\mathcal{A}|$. Moreover, if $X \ne \emptyset$ then $|{\cal B}_n(X)| \geq 1,$ and so if $X \ne \emptyset$ then $0\leq  h(X) \leq \log |\mathcal{A}|.$ We let ${\cal B}(X):=\cup_{n=1}^{\infty}{\cal B}_{n}(X).$ Often ${\cal B}$ is referred to as the language of $X.$ Here and throughout $\log$ is the logarithm to the base $2$.

Recently there has been an increased interest in shifts spaces that have the specification or almost specification property, see \cite{CliTho} and \cite{DasJan1}. Almost specified shifts are a larger class than the class of irreducible sofic shifts, and among many important properties they share, is the fact that they are intrinsically ergodic (there exists a unique measure of maximal entropy).
\begin{definition}
Let $X$ be a shift space. We say that $X$ has the \textbf{specification property} if there exists a positive integer $N \geq 1$ such that
for all $u,v \in {\cal B} (X)$, there exists $\omega \in {\cal B}_N(X)$ with $u\omega v \in {\cal B} (X)$. We say that $X$ has the \textbf{almost specification property} (or $X$ is almost specified) if there exists a positive integer $N \geq 1$ such that
for all $u,v \in {\cal B} (X)$, there exists $\omega \in {\cal B}(X)$, with $|\omega| \leq N,$ and $u\omega v \in {\cal B} (X)$.
\end{definition}
Both the specification property and the almost specification property are useful tools in describing the dynamical properties of certain shift spaces. One of the goals of this paper is to further the understanding of these dynamical properties. With this goal in mind we introduce two new classes of shift spaces, namely boundedly supermultiplicative (BSM) shifts and balanced shifts.
\begin{definition}
A shift $X$ is called \textbf{boundedly supermultiplicative} if there exists $K\geq1$ such that
$$|{\cal B}_m(X)|\cdot|{\cal B}_n(X)| \leq K |{\cal B}_{m+n}(X)|$$
for every $m,n\geq1$. A shift is called \textbf{essentially boundedly supermultiplicative} if there exists
a subshift of the same entropy that is boundedly supermultiplicative.
\end{definition}

\begin{definition}
A shift space $X$ is \textbf{balanced} if there exists $B > 0,$ so that for every word $\omega\in {\cal B}(X)$, we have
 $$\frac{|{\cal B}_{\omega,r}(X)|}{| {\cal B} _r(X)|} \geq B.$$ Where $r\in\mathbb{N}$ and ${\cal B}_{\omega,r}(X)$ is the set of all words of length $r$ that can follow $\omega$.
\end{definition}

In Section $3$ we study the properties of BSM shifts and balanced shifts. We show that every almost specified shift is both BSM and balanced. We also give an example of a shift that is essentially boundedly supermultiplicative but not boundedly supermultiplicative. From the measure theoretic perspective we show that a shift space is balanced if and only if it admits a Gibbs state. The rest of this paper will be focused on a special class of shift space called an $S$-gap shift.
\begin{definition}
Let $S$ be a non-empty subset of $\mathbb{N}$. We define the \textbf{$S$-gap shift} $X(S)\subseteq \{0,1\}^{\mathbb{Z}}$ to be the set of bi-infinite sequences such that the number of consecutive zeros in a sequence is always an element of $S$.
\end{definition}Clearly an $S$-gap shift is a shift space. A typical point in $X(S)$ has the form:
$$x = ...10^{n_{-1}}10^{n_0}10^{n_1}...,$$ where $n_{j}\in S.$ $S$-gap shifts are well studied, for further details we refer the reader to \cite{DasJan1, DasJan2} and the references therein. In Section $4$ we examine the link between $S$-gap shifts and expansions in non-integer bases.  As we will see, using the machinery of expansions in non-integer bases we obtain very simple proofs of dynamical results for $S$-gap shifts. These results include: for any $q\in (0,1)$ there exists an $S$-gap shift which has the specification property and entropy $q$, for any $q\in (0,\log(\frac{1+\sqrt{5}}{2}))$ there exists a continuum of $S$-gap shifts which have the specification property and entropy $q,$ and the set of $q\in(0,1)$ for which there exists a unique $S$-gap shift with entropy $q$ has Hausdorff dimension $1$.

\section{Preliminaries}\label{prelim}

In this section we outline the necessary preliminaries from $S$-gap shifts and expansions in non-integer bases. We begin by recalling the relevant theory from $S$-gap shifts.
\subsection{$S$-gap shifts}
The following proposition is the first step in connecting $S$-gap shifts and expansions in non-integer bases.
\begin{proposition}
\label{entropy prop}
If $S=\{n_0,n_1,...,n_i,...\}$, then
$h(X(S))=\log (\lambda)$, where $\lambda$ is the unique positive solution to
$$1 = \sum_{i \geq 0} \frac{1}{x ^{n_i +1}}$$
\end{proposition}
Proposition \ref{entropy prop} is a well known result. However a detailed proof of this result is almost impossible to find. The most cited reference is \cite{LinMar}, where the statement appears as an exercise. The other widely cited reference is \cite{Spandl}, section 7, but there appears to be an error in the proof (line 5 page 153). For a proof of this theorem in the case when the set $S$ is either finite or cofinite we refer the reader to \cite{DasJan2}. To generalise this result to an arbitrary $S$ the key ingredient is that for any $\epsilon > 0$ there exists $N\in\mathbb{N}$ so that $|h(X(S))-h(X(S \bigcup \{N,N+1,N+2,...\}))| < \epsilon$. This implies Proposition \ref{entropy prop}.

The following proposition follows from the work of \cite{DasJan1}. It links desirable dynamical properties of an $S$-gap shift with certain combinatorial properties of the set $S.$

\begin{proposition}
\label{prop1}
Let $S\subseteq \mathbb{N}$. The following results hold:

\begin{enumerate}
  \item $X(S)$ is a shift of finite type if and only if $S$ is finite or cofinite.
  \item $X(S)$ is almost specified if and only if $\sup_i|n_{i+1}-n_i| < \infty.$
  \item $X(S)$ is mixing if and only if gcd$\{n+1: n \in S \} = 1$.
  \item $X(S)$ has the specification property if and only if $\sup_i|n_{i+1}-n_i| < \infty$ and gcd$\{n+1: n \in S \} = 1$.
\end{enumerate}
\end{proposition}

\subsection{Expansions in non-integer bases}
We now establish the necessary preliminaries from expansions in non-integer bases. Let $\lambda\in(1,2)$ and $I_{\lambda}:=[0,\frac{1}{\lambda-1}].$ Given $x\in \mathbb{R}$ we say that a sequence $(a_{j})_{j=1}^{\infty}\in\{0,1\}^{\mathbb{N}}$ is a \textbf{$\lambda$-expansion} for $x$ if $$\sum_{j=1}^{\infty}\frac{a_{j}}{\lambda^{j}}=x.$$ It is a straightforward exercise to show that $x$ has a $\lambda$-expansion if and only if $x\in I_{\lambda}.$ Expansions in non-integer bases were pioneered in the papers of Parry \cite{Parry} and R\'{e}nyi \cite{Renyi}. One of the appealing features of expansions in non-integer bases is that an $x\in I_{\lambda}$ typically has many $\lambda$-expansions. In fact, for any $\lambda\in(1,2)$ almost every $x\in I_{\lambda}$ has a continuum of $\lambda$-expansions, see \cite{Sid}. Here and hereafter almost every is meant with respect to Lebesgue measure. Given the above, it is natural to introduce the following set $$\Sigma_{\lambda}(x)=\Big\{(a_j)_{j=1}^{\infty}\in \{0,1\}^{\mathbb{N}}: \sum_{j=1}^{\infty}\frac{a_{j}}{\lambda^{j}}=x\Big\}.$$ Let us fix the maps $T_{0}(x)=\lambda x$ and $T_{1}(x)=\lambda x -1$. We now introduce another set $$\Omega_{\lambda}(x)=\Big\{(T_{a_{j}})_{j=1}^{\infty}\in \{T_{0},T_{1}\}^{\mathbb{N}}:  (T_{a_{k}}\circ \cdots \circ T_{a_{1}})(x)\in I_{\lambda}\textrm{ for all } k\in\mathbb{N}\Big\}.$$ In \cite{Baker} the following lemma was shown to hold.
\begin{lemma}
\label{Bijection lemma}
$\textrm{card }\Sigma_{\lambda}(x)=\textrm{card }\Omega_{\lambda}(x).$ Where our bijection identifies $(a_{j})_{j=1}^{\infty}$ with $(T_{a_{j}})_{j=1}^{\infty}.$
\end{lemma}The dynamical interpretation of $\Sigma_{\lambda}(x)$ provided by $\Omega_{\lambda}(x)$ and Lemma \ref{Bijection lemma} will help with our exposition and make our proofs more succinct. Indeed, when it comes to constructing a $\lambda$-expansion of $x\in I_{\lambda}$ with certain desirable combinatorial properties, it is often much easier to achieve this goal by studying the trajectories of $x$ under combinations of $T_{0}$ and $T_{1}$.

Let $\mathcal{S}_{\lambda}:=T_{0}^{-1}(I_{\lambda})\cap T_{1}^{-1}(I_{\lambda})=[\frac{1}{\lambda},\frac{1}{\lambda(\lambda-1)}].$ By definition $\mathcal{S}_{\lambda}$ is the set of $x\in I_{\lambda}$ which are mapped into $I_{\lambda}$ under $T_{0}$ and $T_{1}$. Therefore, by Lemma \ref{Bijection lemma}, $\mathcal{S}_{\lambda}$ is the set of $x\in I_{\lambda}$ which have a choice of digit in the first entry of their $\lambda$-expansion. An $x\in I_{\lambda}$ has a unique $\lambda$-expansion if and only if it is never mapped into $S_{\lambda}$ under a finite concatenation of $T_{0}'s$ and $T_{1}'s$. In what follows we refer to the interval $\mathcal{S}_{\lambda}$ as the \textbf{switch region}.

Both Parry \cite{Parry} and R\'{e}nyi \cite{Renyi} studied a specific dynamical system as a method for generating $\lambda$-expansions. In particular, they studied the expansions generated by the map $G:I_\lambda \to I_\lambda$ where
$$G(x) = \Big\{ \begin{array}{ll}
         T_{0}(x) & \mbox{if $x\in [0,\frac{1}{\lambda})$}\\
        T_{1}(x) & \mbox{if $x\in[\frac{1}{\lambda},\frac{1}{\lambda-1}]$.}\end{array}$$
For each $x\in I_{\lambda},$ by repeatedly applying the map $G$ to $x$ and its successive images we construct an element of $\Omega_{\lambda}(x),$ and therefore by Lemma \ref{Bijection lemma} we have constructed an element of $\Sigma_{\lambda}(x).$ The map $G$ is known as the \textbf{greedy map}, and the $\lambda$-expansion of $x$ generated by $G$ is known as the \textbf{greedy expansion}. The greedy expansion is named as such because it is the $\lambda$-expansion generated by the algorithm which always picks the largest digit possible. We observe that given $x\in[0,\frac{1}{\lambda-1})$, for $n$ sufficiently large $G^{n}(x)\in[0,1).$ Moreover, once $x$ is mapped into $[0,1)$ it is never mapped out, i.e., $G^{m}(x)\in [0,1)$ for all $m\geq n.$

Complementing the notion of a greedy expansion is the notion of a \textbf{lazy expansion}. This is the $\lambda$-expansion generated by the algorithm which always picks the smallest digit possible. This expansion is generated by the map $L:I_{\lambda}\to I_{\lambda}$ where
$$L(x) = \Big\{ \begin{array}{ll}
         T_{0}(x) & \mbox{if $x\in [0,\frac{1}{\lambda(\lambda-1)}]$}\\
        T_{1}(x) & \mbox{if $x\in(\frac{1}{\lambda(\lambda-1)},\frac{1}{\lambda-1}]$.}\end{array}.$$ In a similar way to the greedy map, every $x\in(0,\frac{1}{\lambda-1}]$ is mapped into the interval $(\frac{2-\lambda}{\lambda-1},\frac{1}{\lambda-1}],$ and once $x$ is mapped into this interval it is never mapped out. Notice that the functions $G$ and $L$ only differ in the switch region $\mathcal{S}_{\lambda}.$ This is to be expected as we only ever have a choice of digit when $x$ is mapped into $\mathcal{S}_{\lambda},$ and our greedy expansion chooses the largest digit possible and the lazy expansion the smallest digit possible.

The following straightforward result will be important when it comes to applying the theory of expansions in non-integer bases to $S$-gap shifts.

\begin{lemma}
\label{consecutive zeros remark}
Let $x\in I_{\lambda}$ and $(T_{a_{j}})_{j=1}^{\infty}\in \Omega_{\lambda}(x).$ If there exists $\delta >0$ such that $(T_{a_{k}}\circ \cdots \circ T_{a_{1}})(x)>\delta$ for all $k\in\mathbb{N}$. Then $(T_{a_{j}})_{j=1}^{\infty}$ has an upper bound on the number of consecutive $T_{0}'s$ occurring, and by Lemma \ref{Bijection lemma} the corresponding element of $\Sigma_{\lambda}(x)$ also has an upper bound on the number of consecutive zeros occurring.
\end{lemma}
\begin{proof}
Let $x,(T_{a_{j}})_{j=1}^{\infty}$ and $\delta$ be as above. Let $N\in\mathbb{N}$ be such that $\delta \lambda^{N}>\frac{1}{\lambda-1}.$ We now show that $(T_{a_{j}})_{j=1}^{\infty}$ cannot contain more than $N$ consecutive $T_{0}'s.$

Assume that this is not the case and that there exists $k\in \mathbb{N}$ such that $T_{a_{k+i}}=T_{0}$ for $1\leq i\leq N.$ Using the fact that $T_{0}$ expands distances between distinct points by a factor $\lambda,$ and that $0$ is the unique fixed point of $T_{0},$ we observe the following
\begin{align*}
| (T_{0}^{N}\circ T_{a_{k}}\circ \cdots \circ T_{a_{1}})(x) -0|= |(T_{0}^{N}\circ T_{a_{k}}\circ \cdots \circ T_{a_{1}})(x)- T_{0}^{N}(0)|&= \lambda^{N} |(T_{a_{k}}\circ \cdots \circ T_{a_{1}})(x)- 0|\\
&\geq \lambda^{N}\delta\\
& >\frac{1}{\lambda-1}.
\end{align*}Therefore $(T_{0}^{N}\circ T_{a_{k}}\circ \cdots \circ T_{a_{1}})(x)\notin I_{\lambda}$ and we have our desired contradiction.

\end{proof}
The following corollary is an immediate consequence of Lemma \ref{consecutive zeros remark} and our earlier observation that if $x\in (0,\frac{1}{\lambda-1}]$ then $L$ eventually maps $x$ into $(\frac{2-\lambda}{\lambda-1},\frac{1}{\lambda-1}],$ and once $x$ is mapped into this interval it never leaves.

\begin{corollary}
Let $x\in (0,\frac{1}{\lambda-1}]$. The lazy expansion of $x$ contains a bounded number of consecutive zeros.
\end{corollary}

\section{Boundedly Supermultiplicative and Balanced Shifts}\label{Shifts}

In this section we prove our results about BSM and balanced shifts. We begin by recalling the definitions of a BSM shift and a balanced shift.
\begin{definition}
A shift $X$ is called \textbf{boundedly supermultiplicative} if there exists $K\geq1$ such that $$|{\cal B}_m(X)|\cdot|{\cal B}_n(X)| \leq K |{\cal B}_{m+n}(X)|$$
for every $m,n\geq1$. A shift is called \textbf{essentially boundedly supermultiplicative} if there exists a subshift of the same entropy that is boundedly supermultiplicative.
\end{definition}
\begin{definition}
A shift space $X$ is \textbf{balanced} if there exists $B > 0,$ so that for every word $\omega$, we have $$\frac{|{\cal B}_{\omega,r}(X)|}{| {\cal B} _r(X)|} \geq B.$$ Where $r\in\mathbb{N}$ and ${\cal B}_{\omega,r}(X)$ is the set of all words of length $r$ that can follow $\omega$.
\end{definition}
\subsection{Examples and main properties}
In this subsection we examine some examples and prove some results on BSM and balanced shifts. If turns out that every balanced shift is BSM, and every almost specified shifts is both BSM and balanced.

Our first example shows that there exist essentially boundedly supermultiplicative shifts that are not boundedly supermultiplicative.
\begin{example}
\emph{We start with an alphabet of four letters $A=\{a,b,c,d\}$. Let $X$ be the shift space where the set of forbidden blocks is ${\cal F}=\{ac,ad,bd,ca,cb,da,db\}$. Let $C_1 = \{a,b\}$ and let $C_2 = \{c,d\}$. Given the definition of ${\cal F}$, if a word has a letter from $C_2$, that letter can only be followed by a letter from $C_2$. Also, the block $bc$ can only appear once in an admissible word.\\
 There are $2^n$ words with letters exclusively from $C_1$ and $2^n$ words with letters exclusively from $C_2$. The number of words of length $n$ containing the  word $bc$ is $(n-1)2^{n-2}$. Thus, the total number of words of length $n$ is $2^{n+1}+(n-1)2^{n-2} = (n+7) 2^{n-2}.$ A simple calculation then shows that the shift is not BSM and it has entropy $\log(2)$. $X$ is essentially boundedly supermultiplicative since the subshift corresponding to $C_1$ has entropy $\log(2)$ .}
\end{example}
Next, we examine an example of a sofic shift that is not a subshift of finite type.
\begin{definition}
The even shift is the subshift of $\{0,1\}^{\mathbb{Z}_+}$ comprising all binary sequences
such that there is an even number of 0's between any two successive 1's.
\end{definition}
\begin{example}[The even shift] We show next that the even shift is boundedly supermultiplicative with $K=4$.
Let $E$ denote the even shift. Let ${\cal B}_{m,0}(E)$ denote the set of words of length $m$ ending in $0,$ and ${\cal B}_{0,m}(E)$ denote the set of words of length $m$ starting with $0$. If we have a word of length $m$ ending in $1$ we can get another word of length $m$ by replacing $1$ with $0$. Similarly, if we have a word of length $m$ starting in $1$ we can get another word of length $m$ by replacing $1$ with $0$. Thus we have
$$|{\cal B}_m(E)|\leq2|{\cal B}_{m,0}(E)|\hspace{1cm}\mbox{ and }\hspace{1cm}|{\cal B}_m(E)|\leq2|{\cal B}_{0,m}(E)|$$
for every $m\geq1$.

Now, let $\omega0$ and $0\alpha$ be arbitrary words in ${\cal B}_{m,0}(E)$ and ${\cal B}_{0,n}(E)$ respectively. Let $k$ be the number of $0$'s that $\omega0$ ends in, and let $l$ be the number of $0$'s $0\alpha$ starts with. We define a map $\tau: {\cal B}_{m,0}(E) \ast {\cal B}_{0,m}(E) \to {\cal B}_{m+n}(E)$ as follows: $\tau((\omega0,0\alpha))=\omega00\alpha$ if $k+l$ is even, $\tau((\omega0,0\alpha))=\omega10\alpha$ if $k$ is odd and $l$ is even, and $\tau((\omega0,0\alpha))=\omega01\alpha$ if $k$ is even and $l$ is odd. The map $\tau$ is injective, so we may deduce:
$$|{\cal B}_{m,0}(E)|\cdot|{\cal B}_{0,n}(E)|\leq|{\cal B}_{m+n}(E)|$$
for all $m,n\geq1$.
 Consequently,
$$|{\cal B}_m(E)|\cdot|{\cal B}_n(E)|\leq2|{\cal B}_{m,0}(E)|\cdot2|{\cal B}_{0,n}(E)|\leq4|{\cal B}_{m+n}(E)|.$$
for all $m,n \geq 1$.
In conclusion, $E$ is BSM, with $K=4$.
\end{example}
We now prove that every almost specified shift is BSM and balanced. We begin by showing that every balanced shift is BSM.
\begin{proposition}
\label{BalBSM prop}
Every balanced shift is BSM.
\begin{proof}
If $X$ is balanced, there exists $B>0$ so that, for every $\omega$ in ${\cal B}(X)$ and every $r$ we have $|{\cal B}_{\omega,r}(X)| \geq B|{\cal B}_r(X)|$. Now let $A_{\omega,r}= \{\omega \rho | \rho \in {\cal B}_{\omega,r}(X) \}$. Thus $|A_{\omega,r}| \geq B|{\cal B}_r(X)|$. Then for every $m,n \geq 1$ we have
$$|{\cal B}_{m+n}(X)| = \Sigma_{\omega \in {\cal B}_m(X)}|A_{\omega,n}| \geq \Sigma_{\omega \in {\cal B}_m(X)}B|{\cal B}_n(X)|=B |{\cal B}_m(X)|
|{\cal B}_m(X)|.$$
This proves that $X$ is BSM.
\end{proof}
\end{proposition}

\begin{theorem}
\label{theorem 1}
Every almost specified shift is BSM and balanced.
\begin{proof}
By Proposition \ref{BalBSM prop} it suffices to show that every almost specified shift is balanced. Let $\omega$ be an admissible word. Let $N$ be as in the definition of almost specified. Thus, for our word $\omega$ there exists a word of length at most $N$ connecting $\omega$ with any given word of length $r$. So we can write $$|{\cal B}_{\omega,r}(X)| + |{\cal B}_{\omega,r+1}(X)| + \cdots + |{\cal B}_{\omega,r+N}(X)| \geq |{\cal B}_r(X)|.$$
At the same time we have $$|{\cal B}_{\omega,r}(X)| (1+ |{\cal B}_1(X)| + \cdots + |{\cal B}_N(X)|) \geq |{\cal B}_{\omega,r}(X)| + |{\cal B}_{\omega,r+1}(X)| + |{\cal B}_{\omega,r+N}(X)|.$$
In conclusion
$$\frac{|{\cal B}_{\omega,r}(X)|}{|{\cal B}_r(X)|} \geq (1+ |{\cal B}_1(X)| + \cdots + |{\cal B}_N(X)|)^{-1},$$ and therefore $X$ is balanced.
\end{proof}
\end{theorem}

Given Proposition \ref{BalBSM prop} it is natural to ask whether there exists BSM shifts that are not balanced. This turns out to be the case as we now show. The $S$-gap shift and the $\beta$-shift will be our main examples of BSM shifts that are not balanced. An example of a balanced shift that is not almost specified is at this point an open question.

\begin{example}[$S$-gap shift]
\emph{Let $S$ be a subset of the positive integers so that $S$ has unbounded gaps. Then the corresponding $S$-gap shift is not balanced. Let $S = \{n_0, n_1, ...,n_k,...\}$. Given that $S$ has unbounded gaps, we can find a subset of $S$, call it $W= \{ n_{i_0},n_{i_0+1},n_{i_1},n_{i_1+1},....,n_{i_k},n_{i_k+1}, ...\},$ so that $n_{i_k+1} -n_{i_k} \to \infty$. Consider now $\omega = 10^{n_{i_k}+1}$. Then $|{\cal B}_{\omega,n_{i_{k+1}}-n_{i_k}}(X)|=1$ and so $\frac{|{\cal B}_{\omega,n_{i_{k+1}}-n_{i_k}}(X)|}{|{\cal B}_{n_{i_{k+1}}-n_{i_k}}(X)|} \to 0$ as $k$ approaches infinity. The fact that any $S$-gap shift is BSM is a direct consequence of Lemma 5.3 from \cite{CliTho}.}
\end{example}
\begin{example}[$\beta$-shift]
\emph{In \cite{Renyi} it was shown that any $\beta$-shift is boundedly supermultiplicative. However, when the expansion of $1$ contains arbitrarily long strings of zeros, the $\beta$-shift is not balanced. This follows from a similar argument to that given above in the case of $S$-gap shifts.}
\end{example}
We end this subsection with a result which demonstrates that the entropy of a BSM shift is ``easier" to compute than the entropy of a general shift space.
For more details on boundedly submultiplicative and supermultiplicative sequences and their normalization see \cite{GheRoy}.
\begin{proposition}
Let $X$ be a boundedly supermultiplicative shift. Then for any $n\geq1$,
$$\frac{\log|{\cal B}_n(X)|-\log K}{n}\leq h(X)\leq\frac{\log|{\cal B}_n(X)|}{n}.$$
\end{proposition}The first inequality holds because the sequence $(\frac{|{\cal B}_n(X)|}{K})_{n\geq1}$ is
supermultiplicative. The second inequality holds because the sequence $(|{\cal B}_n(X)|)_{n\geq1}$ is submultiplicative.

\subsection{Measure-theoretic properties of BSM and balanced shifts}
We now turn our attention to some measure theoretic properties of $BSM$ and balanced shifts. We start with the definitions of a Gibbs-like measure and a Gibbs measure on a shift space. Measures with the Gibbs property have been studied in \cite{PetSch} and more recently in \cite{CliTho}.
\begin{definition}
A probability measure $\mu$ is a Gibbs-like state for a shift space $X$ if there exist $c_1>0$ and $c_2>0$ so that, $\forall \omega$, we have
$$ c_1 \leq \mu ([\omega]) | {\cal B}_r(X) | \leq c_2 $$ where $|\omega| = r$.
\end{definition}
\begin{definition}
A probability measure $\mu$ is a Gibbs state for a shift space $X$ if there exist $c_1$ and $c_2$ strictly positive, so that, $\forall \omega$, we have
$$ c_1 \leq \mu ([\omega])  2^{r h_X} \leq c_2 $$
where $|\omega| = r$ and $h(X)$ is the entropy of $X$.
\end{definition}

\begin{proposition}
Let $X$ be a shift space. The following statements hold.
\begin{enumerate}
  \item $X$ is BSM if and only if there exist positive constants $c_1$ and $c_2$ so that for every $n \geq 1$, we have
\begin{equation}
\label{BSMequivalence}
 c_1 \leq \frac{2^{n h_X}}{| {\cal B}_n(X) |} \leq c_2 .
\end{equation}
Thus, in a BSM shift space the notions of Gibbs state and Gibbs-like state are equivalent.
  \item If $X$ admits a Gibbs state then $X$ is BSM. In particular, any Gibbs state is a Gibbs-like state.
\end{enumerate}
\begin{proof}
 1) Let us assume first that $X$ is BSM. We recall the well known fact that the entropy of $X$ can be expressed as $$h(X)= \inf_{n \geq 1} \frac{1}{n} \log |{\cal B}_n(X)|.$$ This implies that for every $n \geq 1$ we have $h(X) \leq \frac{1}{n} \log |{\cal B}_n(X)|.$ Therefore $2^{n h(X)} \leq |{\cal B}_n(X)|$ for every $n \geq 1$ and we can take $c_2 =1$.

To construct our $c_1$ we observe that Proposition 3.10 implies that for every $n \geq 1$
$$\frac{\log|{\cal B}_n(X)|-\log K}{n}\leq h(X).$$
Therefore $\frac{|{\cal B}_n(X)|}{K} \leq 2^{n h(X)}$ for every $n \geq 1$ and we may take $c_1 = \frac{1}{K}$.

Going in the other direction, assume (\ref{BSMequivalence}) holds. For every $m,n \geq 1$ we have
$$c_1 \leq \frac{2^{m h(X)}}{| {\cal B}_m(X) |} \leq c_2 \textrm{ and } c_1 \leq \frac{2^{n h(X)}}{| {\cal B}_n(X) |} \leq c_2.$$
Multiplying these inequalities we get
\begin{equation}
\label{3131}
c_1^2 \leq \frac{2^{(m+n)h(X)}}{| {\cal B}_m(X) || {\cal B}_n(X) |} \leq c_2^2.
\end{equation}
By (\ref{BSMequivalence}) we have  $2^{(m+n) h(X)} \leq c_2 | {\cal B}_{m+n}(X) |$, so
\begin{equation}
\label{3132}
c_1^2 \leq \frac{c_2 | {\cal B}_{m+n}(X) |}{| {\cal B}_m(X) || {\cal B}_n(X) |}.
\end{equation}
Using (\ref{3132}) we get
$$| {\cal B}_m(X) || {\cal B}_n(X) | \leq \frac{c_2}{c_1^2} | {\cal B}_{m+n}(X) |.$$
This shows that $X$ is BSM. Finally, the fact that the notions of Gibbs state and Gibbs-like state are equivalent in a BSM shift space is a straightforward consequence of their definitions.\\
\\
2) Assume that a Gibbs state $\mu$ exists on $X$. Fix $r \geq 1$ and let $\omega$ be an arbitrary admissible word of length $r$. By the definition of a Gibbs state, there exist $c_1$ and $c_2$ so that
$$ c_1 \leq \mu ([\omega])  2^{r h(X)} \leq c_2 .$$
Summing over all words of length $r$ we obtain
$$c_1 | {\cal B}_r(X) | \leq \sum_{|\omega|=r} \mu ([\omega])  2^{r h(X)} \leq c_2 | {\cal B}_r(X) |. $$
This is equivalent to
$$c_1 | {\cal B}_r(X) | \leq  2^{r h(X)} \leq c_2 | {\cal B}_r(X) |. $$
Which by 1) implies that $X$ is BSM. The fact that any Gibbs state is a Gibbs-like state also follows from 1).
\end{proof}
\end{proposition}
For our next result we make use of some results from \cite{Walters} about probability measures on compact metric spaces. Recall that our shifts spaces are compact metric spaces. Given a compact metric space $X,$ by $M(X)$ we denote the set of all probability measures on $X$. We endow $M(X)$ with the $weak^*$ topology. If $X$ is a compact metric space, then $M(X)$ is compact with the $weak^*$ topology.
\begin{theorem}
\label{State balanced thm}
A shift space $X$ admits a Gibbs state if and only if it is balanced.
\begin{proof}
First we show how to construct a Gibbs state on the balanced shift space $X$. For every $r \geq 1$, and for every admissible word $\omega$ of length $r$, we define a point-mass probability measure $\mu_r$ as follows. Choose $\alpha$ an infinite word in $[\omega]$ and define $\mu_r(\{\alpha\})=\frac{1}{| {\cal B}_r(X) |}$. Now, we have
$\mu_r([\omega]) = \frac{1}{ | {\cal B}_r(X) |},$ and $\mu_r([\omega\tau])$ is $0$ if it doesn't contain $\alpha,$ or is the full measure of the cylinder otherwise.

Given $\omega$ and $r$, for every $k \geq 1$, we have
$$\mu_{r+k}([\omega]) = \mu_{r+k}(\bigcup_{\substack{|\alpha|=k\\ \omega\alpha\in \mathcal{B}_{r+k}(X)}}[\omega\alpha]) =\sum_{\substack{|\alpha|=k\\ \omega\alpha\in \mathcal{B}_{r+k}(X)}}\mu ([\omega\alpha])=|{\cal B}_{\omega,k}(X) |\frac{1}{|{\cal B}_{r+k}(X) |}.$$\\
Since $X$ is balanced, there exists $B > 0$, independent of $\omega$ and $k$, so that $B \leq \frac{|{\cal B}_{\omega,k}(X)|}{| {\cal B} _k(X)|}$. Thus
$$B | {\cal B} _k(X)| \leq |{\cal B}_{\omega,k}(X)| \leq | {\cal B} _k(X)|.$$ Hence
$$\frac{B | {\cal B} _k(X)|} {| {\cal B} _{r+k}(X)|} \leq \mu_{r+k}([\omega]) \leq \frac{ | {\cal B} _k(X)|} {| {\cal B} _{r+k}(X)|}.$$
By Proposition \ref{BalBSM prop} every balanced shift is BSM. This means that there exists $K \geq 1$ so that
$$|{\cal B}_{r+k}(X)| \leq |{\cal B}_r(X)|\cdot|{\cal B}_k(X)| \leq K |{\cal B}_{r+k}(X)|.$$
Thus
$$\frac{1}{|{\cal B}_r(X)|} \leq \frac{|{\cal B}_k(X)|}{|{\cal B}_{r+k}(X)|} \leq \frac{K}{|{\cal B}_r(X)|}.$$
Choosing $c_1=B$ and $c_2=K$, we conclude that there exists $c_1$ and $c_2$ so that
\begin{equation}
\label{weak^*}
\frac{c_1} {| {\cal B}_r(X) |} \leq \mu_{r+k}([\omega]) \leq \frac{c_2} {| {\cal B}_r(X) |}.
\end{equation}
At this point, let us choose a convergent subsequence for our sequence of measures, call it $\{\mu_{p_l}\}_{l \geq 1}$. Let $\mu$ be the weak$^*$ limit of this subsequence.

Since our alphabet is finite, the cylinder $[\omega]$ is both open and closed for any finite word $\omega$.
Let $\omega$ now be an arbitrary word and let $|\omega|=r$. For any $l$ so that $p_l \geq r$, we have
\begin{equation}
\frac{c_1} {| {\cal B}_r(X) |} \leq \mu_{p_l}([\omega]) \leq \frac{c_2} {| {\cal B}_r(X) |}.
\end{equation}
Letting $l$ go to infinity we obtain
\begin{equation}
\frac{c_1} {| {\cal B}_r(X) |} \leq \mu([\omega]) \leq \frac{c_2} {| {\cal B}_r(X) |}.
\end{equation}
This is a consequence of $[\omega]$ being open and closed. This shows that $\mu$ is a Gibbs-like state and thus, using Proposition 3.13 2), a Gibbs state (since $X$ is BSM).

Going in the other direction, suppose $X$ admits a Gibbs state $\mu$. Let $\omega$ an arbitrary word of length $r,$ and let $k\geq 1$ be an integer. By Proposition 3.13 we know that $\mu$ is also a Gibbs like state, therefore $$\frac{c_1}{|{\cal B}_r(X)|} \leq \mu ([\omega]) = \mu (\bigcup_{\substack{|\alpha|=k\\ \omega\alpha\in \mathcal{B}_{r+k}(X)}}[\omega\alpha]) =\sum_{\substack{|\alpha|=k\\ \omega\alpha\in \mathcal{B}_{r+k}(X)}}\mu ([\omega\alpha])\leq c_2 |{\cal B}_{\omega,k}(X) |\frac{1}{|{\cal B}_{r+k}(X) |}.$$\\
Hence
\begin{equation}
\label{3133}
c_1 \leq c_2 |{\cal B}_{\omega,k}(X) ||{\cal B}_r(X)| \frac{1}{|{\cal B}_{r+k}(X) |}.
\end{equation}
Using Proposition 3.13 2), $X$ is BSM, and so, there exists $K \geq 1$ so that
$$|{\cal B}_r(X)|\cdot|{\cal B}_k(X)| \leq K |{\cal B}_{r+k}(X)|.$$
Thus
\begin{equation}
\label{3134}
\frac{|{\cal B}_r(X)|}{|{\cal B}_{r+k}(X)|} \leq \frac{K}{|{\cal B}_k(X)|}.
\end{equation}
Using (\ref{3133}) and (\ref{3134}) we get:
$$c_1 \leq c_2 K \frac{|{\cal B}_{\omega,k}(X) |}{|{\cal B}_k(X) |}.$$
Letting $d_1 = \frac{c_1}{c_2 K}$, we have
$$d_1 \leq  \frac{|{\cal B}_{\omega,k}(X) |}{|{\cal B}_k(X) |}.$$
This shows that $X$ is balanced and finishes the proof.
\end{proof}
\end{theorem}
The following corollaries are an immediate consequence of Theorem \ref{State balanced thm}.
\begin{corollary}
Suppose that the $\beta$-expansion of $1$ contains arbitrarily long sequences of zeros. Then the $\beta$-shift does not admit a Gibbs state.
\end{corollary}
\begin{corollary}
If $S$ is a subset of the positive integers so that $S$ has unbounded gaps. Then the corresponding $S$-gap shift does not admit a Gibbs state.
\end{corollary}





\section{$S$-gap shifts and expansions in non-integer bases}

In this section we further examine the connection between S-gap shifts and expansions in non-integer bases. To begin with, let us recall the content of Proposition \ref{entropy prop}. Given an $S$-gap shift defined by a set $S=\{n_{0},n_{1},\ldots,n_{i},\ldots\}\subseteq \mathbb{N},$ then the entropy of the shift is $\log \lambda$ where $\lambda$ is the unique solution to the equation

\begin{equation}
\label{entropy equation}
\sum_{n_{i}\in S}\frac{1}{\lambda^{n_{i}+1}}=1.
\end{equation}
Going in the opposite direction, suppose we have a sequence $(a_{j})_{j=1}^{\infty}\in\{0,1\}^{\mathbb{N}}$  and $\lambda\in(1,2)$ such that
\begin{equation}
\label{1 expansion}
\sum_{j=1}^{\infty}\frac{a_{j}}{\lambda^{j}}=1.
\end{equation} Then if we let $S=\{j-1:a_{j}=1\},$ the entropy of the associated $S$-gap shift is $\log \lambda.$ Combining the above statements the following proposition holds.

\begin{proposition}
\label{Bijection proposition}
The map sending $(a_{j})_{j=1}^{\infty}$ to the $S$-gap shift defined by $S:=\{j-1:a_{j}=1\}$ is a bijection between sequences satisfying (\ref{1 expansion}) and $S$-gap shifts with entropy $\log \lambda.$
\end{proposition}As we will see, Proposition \ref{Bijection proposition} will be fundamental in proving all of our results connecting $S$-gap shifts and expansions in non-integer bases. What remains of this section will be split into two parts. We begin by proving several results for $S$-gap shifts that depend on an Proposition \ref{Bijection proposition}. In the second part we recall several fundamental results from expansions in non-integer bases. We then show that as a consequence of Proposition \ref{Bijection proposition} they yield equally as appealing analogues in the setting of $S$-gap shifts.

\begin{proposition}
\label{prop2}
For any $q\in (0,1)$ there exists an $S$-gap shift which has the specification property and entropy $q$.
\end{proposition}
\begin{proof}

The case where $q\in (0,\log(\frac{1+\sqrt{5}}{2}))$ will follow from a stronger result proved in the following proposition. As such we just consider $q\in [\log(\frac{1+\sqrt{5}}{2}),1).$ By Proposition \ref{prop1} and Proposition \ref{Bijection proposition}, it suffices to show that for any $\lambda\in[\frac{1+\sqrt{5}}{2},2)$ there exists a $\lambda$-expansion of $1$ which contains two consecutive ones and the number of consecutive zeros is bounded.

For $\lambda=\frac{1+\sqrt{5}}{2}$ this is a simple consequence of the fact that $\sum_{j=2}^{\infty}\lambda^{-j}=1$. Now let $\lambda\in (\frac{1+\sqrt{5}}{2},2).$ We consider the point $T_{1}^{2}(1)$. It is easy to show that $T_{1}^{2}(1)\in (0,\frac{1}{\lambda-1})$. We then apply the lazy map to $T_{1}^{2}(1).$ Under this map $T_{1}^{2}(1)$ is eventually mapped into $(\frac{2-\lambda}{\lambda-1},\frac{1}{\lambda-1}]$ and stays there. By Lemma \ref{consecutive zeros remark} there exists a $\lambda$-expansion of $1$ with the desired properties.
\end{proof}For $q\in (0,\log(\frac{1+\sqrt{5}}{2}))$ we can prove the following stronger statement.

\begin{theorem}
\label{continuum prop}
For any $q\in (0,\log(\frac{1+\sqrt{5}}{2}))$ there exists a continuum of $S$-gap shifts which have the specification property and entropy $q.$
\end{theorem}
\begin{proof}
By a similar argument to that given in Proposition \ref{prop2} it suffices to show that for any $\lambda\in(1,\frac{1+\sqrt{5}}{2})$ there exists a continuum of $\lambda$-expansions of $1$ which contain two consecutive ones and a bounded number of consecutive zeros.

Fix $\lambda \in (1,\frac{1+\sqrt{5}}{2})$. Let us consider the interval $[\frac{1}{\lambda^{2}-1},\frac{\lambda}{\lambda^{2}-1}].$  The significance of the endpoints $\frac{1}{\lambda^{2}-1}$ and $\frac{\lambda}{\lambda^{2}-1}$ is that $$T_{0}\Big(\frac{1}{\lambda^{2}-1}\Big)=\frac{\lambda}{\lambda^{2}-1} \textrm{ and } T_{1}\Big(\frac{\lambda}{\lambda^{2}-1}\Big)=\frac{1}{\lambda^{2}-1}.$$ In other words they form a period two orbit. Moreover, by the monotonicity of the maps $T_{0}$ and $T_{1},$ it follows that any $x\in(0,\frac{1}{\lambda-1})$ is eventually mapped into $[\frac{1}{\lambda^{2}-1},\frac{\lambda}{\lambda^{2}-1}]$ by repeatedly applying $T_{0}$ or $T_{1}.$  It is easy to show that $\frac{1}{\lambda^{2}-1}>\frac{1}{\lambda}$ and  $\frac{\lambda}{\lambda^{2}-1}<\frac{1}{\lambda(\lambda-1)}$ for all  $\lambda \in (1,\frac{1+\sqrt{5}}{2}).$ As a consequence of this $[\frac{1}{\lambda^{2}-1},\frac{\lambda}{\lambda^{2}-1}]\subsetneq \mathcal{S}_{\lambda}.$ The above observations will be critical in building our continuum of $\lambda$-expansions.

We now construct our set of $\lambda$-expansions. Each of these $\lambda$-expansions will share the same initial block of zeros and ones. This initial block will contain two consecutive ones. Let us start by constructing this initial block. For $\lambda\in(1,\frac{1+\sqrt{5}}{2})$ we have $1<\frac{\lambda}{\lambda^{2}-1}.$ We repeatedly apply $T_{0}$ to $1$ until it is mapped into $(\frac{\lambda}{\lambda^{2}-1},\frac{1}{\lambda-1}).$ Let $n\in\mathbb{N}$ be such that $T^{n}_{0}(1)\in (\frac{\lambda}{\lambda^{2}-1},\frac{1}{\lambda-1}).$ By the monotonicity of the map $T_{1}$ and using the fact that $T_{1}^{2}(\frac{\lambda}{\lambda^{2}-1})>0,$ we deduce that $(T_{1}^{2}\circ T^{n}_{0})(1)\in (0,\frac{1}{\lambda-1}).$ The sequence $((0)^{n},1,1)$ will be the initial block shared by of all of the $\lambda$-expansions we construct. Where $(\cdot)^{n}$ denotes the $n$-fold concatenation.

We now show that there exists a continuum of $\lambda$-expansions of $1$ beginning with $((0)^{n},1,1)$ and for which the number of consecutive zeros is bounded. We have that $(T_{1}^{2}\circ T^{n}_{0})(1)\in (0,\frac{1}{\lambda-1}),$ by repeatedly applying either $T_{0}$ or $T_{1}$ it is clear that $(T_{1}^{2}\circ T^{n}_{0})(1)$ is eventually mapped into  $[\frac{1}{\lambda^{2}-1},\frac{\lambda}{\lambda^{2}-1}].$ Let $i$ be the unique element of $\{0,1\}$ and $m\in\mathbb{N}$ be the unique minimal natural number such that $$(T_{i}^{m}\circ T_{1}^{2}\circ T^{n}_{0})(1)\in\Big[\frac{1}{\lambda^{2}-1},\frac{\lambda}{\lambda^{2}-1}\Big].$$ As $[\frac{1}{\lambda^{2}-1},\frac{\lambda}{\lambda^{2}-1}]\subsetneq \mathcal{S}_{\lambda}$ both $$(T_{0}\circ T_{i}^{m}\circ T_{1}^{2}\circ T^{n}_{0})(1)\in \Big(0,\frac{1}{\lambda-1}\Big)\textrm{ and }(T_{1}\circ T_{i}^{m}\circ T_{1}^{2}\circ T^{n}_{0})(1)\in \Big(0,\frac{1}{\lambda-1}\Big).$$ In fact both of these images are contained within the interval $[T_{1}(\frac{1}{\lambda^{2}-1}),T_{0}(\frac{\lambda}{\lambda^{2}-1})]$. There exists $m',m''\in\mathbb{N}$ such that  $$(T_{1}^{m'}\circ T_{0}\circ T_{i}^{m}\circ T_{1}^{2}\circ T^{n}_{0})(1)\in\Big[\frac{1}{\lambda^{2}-1},\frac{\lambda}{\lambda^{2}-1}\Big]\textrm{ and }(T_{0}^{m''}\circ T_{1}\circ T_{i}^{m}\circ T_{1}^{2}\circ T^{n}_{0})(1)\in\Big[\frac{1}{\lambda^{2}-1},\frac{\lambda}{\lambda^{2}-1}\Big].$$ As such we can apply $T_{0}$ and $T_{1}$ to both $(T_{1}^{m'}\circ T_{0}\circ T_{i}^{m}\circ T_{1}^{2}\circ T^{n}_{0})(1)$ and $(T_{0}^{m''}\circ T_{1}\circ T_{i}^{m}\circ T_{1}^{2}\circ T^{n}_{0})(1).$ Again all of these images are contained in
$[T_{1}(\frac{1}{\lambda^{2}-1}),T_{0}(\frac{\lambda}{\lambda^{2}-1})]$ and we can repeat the above steps. Moreover we can repeat this entire procedure indefinitely, with the all of the orbits of $1$ never leaving $[T_{1}(\frac{1}{\lambda^{2}-1}),T_{0}(\frac{\lambda}{\lambda^{2}-1})].$ It is clear that by repeatedly applying this procedure we generate a continuum of elements of $\Omega_{\lambda}(1)$ and therefore by Lemma \ref{Bijection lemma} a continuum of elements of $\Sigma_{\lambda}(1)$. To see that each of these $\lambda$-expansions contain a bounded number of consecutive zeros, we highlight the fact each element of our construction maps $1$ into $[T_{1}(\frac{1}{\lambda^{2}-1}),T_{0}(\frac{\lambda}{\lambda^{2}-1})]$ after a finite number of steps, and once it is mapped into this interval it never leaves never leaves it. Since $T_{1}(\frac{1}{\lambda^{2}-1})>0$ we satisfy the hypothesis of Lemma \ref{consecutive zeros remark} and we can conclude our result.

\end{proof}
Adapting the proof of Theorem 1.3 from \cite{Baker2}, it can be shown that for $\lambda\in(1,\frac{1+\sqrt{5}}{2})$ the set of $\lambda$-expansions of $1$ containing two consecutive ones and a bounded number of consecutive zeros is a set of positive Hausdorff dimension when $\{0,1\}^{\mathbb{N}}$ is endowed with the metric stated in the introduction. The value $\log(\frac{1+\sqrt{5}}{2})$ appearing in Theorem \ref{continuum prop} is in fact optimal. As the following results shows.

\begin{proposition}
\label{golden prop}
The number of $S$-gap shifts with entropy $\log (\frac{1+\sqrt{5}}{2})$ is countably infinite. Moreover, an infinite subset of these $S$-gap shifts satisfy the specification property.
\end{proposition}
\begin{proof}
In \cite{EHJ} Erd\H{o}s, Horv\'ath, Jo\'o showed that
$$\Sigma_{\frac{1+\sqrt{5}}{2}}(1)=\Big\{((10)^{\infty}), ((10)^{n},1,1,(0)^{\infty}), ((10)^{n},0,(1)^{\infty})\textrm{ for some } n\in\mathbb{N}\Big\}.$$ Where in the above $(\cdot)^{\infty}$ denotes the infinite concatenation. This set is clearly countable, so by Proposition \ref{Bijection proposition} we can deduce the first part of our result. To determine the second part, we observe that each expansion of the form $((10)^{n},0,(1)^{\infty})$ defines a unique $S$-gap shift which has the specification property. This follows from the same argument given in Proposition \ref{prop2} and Theorem \ref{continuum prop}.
\end{proof}
Proposition \ref{golden prop} in fact holds for the logarithm of any multinacci number. Multinacci numbers are defined to be the unique real solutions to equations of the form $x^{n}=x^{n-1}+\cdots +x +1$ with modulus strictly greater than $1$.

We now state a result which describes the situation for a Lebesgue generic $q\in(0,1).$ In \cite{Sch} Schmeling showed that for almost every $\lambda\in(1,2)$ the orbit of $1$ under the greedy map $G$ visited any interval in $[0,1)$ with positive frequency. In fact Schmeling proved something much stronger than this, but for now we restrict ourselves to this weaker statement. This result implies that the orbit of $1$ visits the switch region $S_{\lambda}$ with positive frequency. Therefore the set $\Sigma_{\lambda}(1)$ is always at least countably infinite for almost every $\lambda\in(1,2)$. The following statement is a consequence of this.

\begin{theorem}
\label{almost everywhere theorem}
For almost every $q\in(0,1)$ there exists a countable infinite of $S$-gap shifts which have the specification property and entropy $q$.
\end{theorem}
We do not include the details of the proof of this theorem. By Schmeling's result we know that almost every $q\in(0,1)$ must have a countable infinite of $S$-gap shifts with entropy $q$. Showing that we may further assert that they satisfy the specification property relies on combinatorial arguments analogous to those given in Theorem \ref{continuum prop}.

\subsection{Univoque bases and entropies achieved by a unique $S$-gap shift}
It is a consequence of Proposition \ref{Bijection proposition} that if $1$ has a unique $\lambda$-expansion then there exists a unique $S$-gap shift with entropy $\log \lambda.$ If $\lambda\in(1,2)$ is such that $1$ has a unique $\lambda$-expansion then $\lambda$ is called a \textbf{univoque base}. We denote the set of univoque bases by $U$. The set $U$ has been studied thoroughly from the perspective of expansions in non-integer bases, see \cite{DaKa,DvKom,KomLor, KomLor2} and the references therein.
Two of the standout results regarding the set $U$ are the following due to Dar\'{o}czy and  I. Katai \cite{DaKa} and Komornik and Loreti \cite{KomLor}.
\begin{theorem}[Dar\'{o}czy and  I. Katai]
$U$ has Hausdorff dimension $1$.
\end{theorem}$U$ has Lebesgue measure zero by Theorem \ref{almost everywhere theorem}.
\begin{theorem}[Komornik and Loreti]
\label{KomLor theorem}
The smallest element of $U$ is $\lambda_{KL}\approx 1.787$, where $\lambda_{KL}$ is the unique solution to the equation $$\sum_{j=1}^{\infty}\frac{\omega_{j}}{\lambda^{j}}=1.$$ Here $(\omega_{j})_{j=0}^{\infty}$ is the Thue-Morse sequence.
\end{theorem}
The Thue-Morse sequence is defined iteratively as follows: $\omega_{0}=0$ and if $\omega_{i}$ is already defined for some $i\geq 0$ then $\omega_{2i}=\omega_{i}$ and $\omega_{2i+1}=1-\omega_{i}.$ See \cite{AllShall} for more on the Thue-Morse sequence. In \cite{AllCos} it was shown that $\lambda_{KL}$ was transcendental.

Rephrasing these theorems in terms of $S$-gap shifts via Proposition \ref{Bijection proposition}, the following theorems are immediate.

\begin{theorem}
\label{DaKa S gap theorem}
The set of $q\in (0,1)$ for which there exists a unique $S$-gap shift with entropy $q$ has Hausdorff dimension $1$.
\end{theorem}
\begin{theorem}
\label{KomLor S gap theorem}
The smallest value of $q\in (0,1)$ for which there exists a unique $S$-gap shift with entropy $q$ is $\log \lambda_{KL}\approx 0.580.$
\end{theorem}
Theorem \ref{KomLor S gap theorem} is an obvious application of Proposition \ref{Bijection proposition} and Theorem \ref{KomLor theorem}. To prove Theorem \ref{DaKa S gap theorem} we have to be slightly more careful and appeal to the fact that $\log$ is a bi-Lipschitz map on $[\lambda_{KL},2).$

The results we have stated above are just two of the consequences of Proposition \ref{Bijection proposition} and the work of other authors on the set $U$. There are many more results which would follow immediately. We have restricted ourselves to Theorem \ref{DaKa S gap theorem} and Theorem \ref{KomLor S gap theorem} as a result of personal preference.

\medskip\noindent {\bf Acknowledgments.} The authors are grateful to Karl Petersen and Vaughn Climenhaga for their useful remarks.






\begin{thebibliography}{100}
\bibitem{AllCos} J,-P. Allouche and M. Cosnard, \textit{The Komornik-Loreti constant is transcendental}, Amer. Math. Monthly 107 (2000), no. 5, 448--449.
\bibitem{AllShall} J,-P. Allouche and J. Shallit, \textit{The ubiquitous Prouhet-Thue-Morse sequence}, in C. Ding, T. Helleseth, and H. Niederreiter, eds., Sequences and their applications: Proceedings of SETA '98, Springer-Verlag, 1999, pp. 1--16.
\bibitem{Baker} S. Baker, \textit{Generalised golden ratios over integer alphabets,} Integers 14(2014), Paper No A15.
\bibitem{Baker2} S. Baker, \textit{The growth rate and dimension theory of beta-expansions,} Fund. Math. 219 (2012), 271--285.
\bibitem{CliTho} V. Climenhaga and D. Thompson, \textit{Intrinsic ergodicity beyond specification: $\beta$-shifts, S-gap shifts, and their factors}, Israel Journal of Mathematics, 192 (2012), 785-817.
\bibitem{DaKa}  Z. Dar\'{o}czy, I. Katai, \textit{On the structure of univoque numbers},  Publ. Math. Debrecen 46 (1995), no. 3-4, 385–-408.
\bibitem{DvKom} M. de Vries, V. Komornik, \textit{Unique expansions of real numbers},  Adv. Math. 221 (2009), no. 2, 390--427.
\bibitem{DasJan1} D. Dastjerdi and S. Jangjoo, \textit{Dynamics and topology of S-gap shifts}, Topology and Applications 159 (2012) no. 10-11, 2654-2661.
\bibitem{DasJan2}  D. Dastjerdi and S. Jangjoo, \textit{Sofic S-gap Shifts, Entropy Function and Bowen-Franks Groups}, arXiv:1108.3414
\bibitem{EHJ} P. Erd\H{o}s, M. Horv\'ath, I. Jo\'o. \textit{On the uniqueness of the expansions $1 =\sum_{i=1}^{\infty}q^{-n_{i}},$}  Acta Math. Hungar. 58 (1991), no. 3-4, 333--342.
\bibitem{GheRoy} A. E. Ghenciu, M. Roy, \textit{Gibbs States for non-irreducible Markov shifts}, Fund. Math. 221 (2013), 231-265.
\bibitem{KomLor} V. Komornik and P. Loreti, \textit{Unique developments in non-integer bases,} Amer. Math. Monthly 105 (1998), no. 7, 636--639.
\bibitem{KomLor2} V. Komornik and P. Loreti, \textit{Subexpansions, superexpansions and uniqueness properties in non-integer bases,} Period. Math. Hungar. 44 (2002), no. 2, 197--218.
\bibitem{LinMar} D. Lind and B. Marcus, \textit{An introduction to symbolic dynamics and coding}, Cambridge University Press, 1995.
\bibitem{Parry} W. Parry, \textit{On the $\beta$-expansions of real numbers}, Acta Math. Acad. Sci. Hung. {\bf 11} (1960) 401--416.
\bibitem{PetSch} K. Petersen, K. Schmidt, \textit{Symmetric Gibbs measures}, Trans. Amer. Math. Soc. 349 (1997), 2775-2811
\bibitem{Renyi} A. R\'{e}nyi, \textit{Representations for real numbers and their ergodic properties}, Acta Math. Acad. Sci. Hung. {\bf 8} (1957) 477--493.
\bibitem{Sch} J. Schmeling, \textit{Symbolic dynamics for $\beta$-shifts and self-normal numbers}, Ergodic Theory Dynam. Systems 17 (1997), no. 3, 675--694.
\bibitem{Sid} N. Sidorov, \textit{Almost every number has a continuum of $\beta$-expansions}, Amer. Math. Monthly 110 (2003), no. 9, 838--842.
\bibitem{Spandl} C. Spandl, \textit{Computing the Topolgical Entropy of Shifts}, Math. Log. Quart. 53 (2007), no 4/5, 493-510.
\bibitem{Walters} P. Walters, \textit{An Introduction to Ergodic Theory}, Springer, 1982.
\end{thebibliography}
\end{document}